\documentclass{article}
\usepackage{jones}
\usepackage{fullpage}
\usepackage{tikz}
\usepackage{pgfplots}
\usepgfplotslibrary{groupplots}
\pgfplotsset{compat=1.18}

\title{\Large The Grothendieck Constant is Strictly Larger than Davie--Reeds' Bound}
\author{Chris Jones\\UC Davis\\\href{mailto:chijones@ucdavis.edu}{chijones@ucdavis.edu} \and Giulio Malavolta\\Bocconi University\\\href{mailto:giulio.malavolta@unibocconi.it}{giulio.malavolta@unibocconi.it}}
\date{}

\DeclareMathOperator{\He}{He}

\begin{document}

\maketitle

\begin{abstract}
    The Grothendieck constant $K_G$ is a fundamental quantity in functional analysis, with important connections to quantum information, combinatorial optimization, and the geometry of Banach spaces. Despite decades of study, the value of $K_G$ is unknown. The best known lower bound on $K_G$ was obtained independently by Davie and Reeds in the 1980s. In this paper we show that their bound is not optimal. We prove that $K_G \ge K_{DR}+10^{-12}$, where $K_{DR}$ denotes the Davie–Reeds lower bound.

    Our argument is based on a perturbative analysis of the Davie–Reeds operator. We show that every near-extremizer for the Davie–Reeds problem has $\Omega(1)$ weight on its degree-3 Hermite coefficients, and therefore introducing a small cubic perturbation increases the integrality gap of the operator.
\end{abstract}

\section{Introduction}

The (real) \emph{Grothendieck constant} is defined by
\begin{equation}\label{eq:kg-ratio}
    K_G := \sup_{n \in \N}\sup_{A \in \R^{n \times n}} \frac{\displaystyle\max_{\substack{\norm{x_i}_2 = 1, \norm{y_j}_2 = 1}}\sum_{i,j = 1}^n A_{ij} \ip{x_i, y_j}}{\displaystyle\max_{\substack{x_i \in \{\pm1\}, y_j \in \{\pm1 \}}} \sum_{i,j = 1}^n A_{ij} x_i y_j}\,.
\end{equation}

Originally studied by Grothendieck \cite{grothendieck1956resume} in the context of functional analysis,
the constant $K_G$ also plays an important role in quantum mechanics and computer science. For instance, $K_G$ characterizes the maximal violation of Bell inequalities for 2-player XOR games via the celebrated theorem of Tsirelson \cite{tsirelson1980quantum}. In theoretical computer science, 
\cref{eq:kg-ratio} is interpreted as the integrality gap of a certain semidefinite program (SDP),
which has intimate connections to approximation algorithms \cite{alon2004approximating}, regularity lemmas \cite{szemeredi1975regular,frieze1996regularity}, and the Unique Games conjecture \cite{raghavendra2009towards,khot2011grothendieck}. We do not attempt to list all connections of $K_G$ here and instead we refer the reader to the surveys by Pisier \cite{pisier2012grothendieck} and by Khot and Naor \cite{khot2011grothendieck}.

The exact value of $K_G$ is still unknown. The best current numerical bounds are approximately $1.6769 \leq K_G \leq 1.7823$.
The upper bound $K_G \leq \frac{\pi}{2\ln(1+\sqrt{2})} \approx 1.7823$ was proven by Krivine \cite{krivine1977sur}.
While this is the best upper bound on $K_G$ numerically speaking, in a breakthrough work, Braverman et al.\ \cite{braverman2013grothendieck} showed that Krivine's bound could be lowered by a positive constant $\eps > 0$ (the $\eps$ coming out of their proof is so minuscule that they choose not to state an explicit numerical value).

The lower bound $K_G \geq K_{DR} \approx 1.6769$ was proven independently by Davie \cite{davie1984lower} and Reeds \cite{reeds1991new}.
Analogous to the work of Braverman et al.\ relative to the Krivine bound,
our contribution is to give a small but positive improvement to the Davie--Reeds bound, proving that the lower bound for $K_G$ can be improved by a constant $\eps$.
This marks the first movement of the $K_G$ lower bound since the 1980s.

\begin{theorem}\label{thm:new-lb}
    $K_G \geq K_{DR} + 10^{-12}$.
\end{theorem}

\paragraph{Concurrent work.}
In a concurrent work, Heilman \cite{heilman2026lower} has proven $K_G \geq K_{DR} + 10^{-26}$ using the same perturbation strategy that we use here, with a different analysis.

\subsection{What are hard instances for the Grothendieck problem?}

To prove lower bounds on $K_G$, we need to identify matrices $A \in \R^{n \times n}$ that maximize the ratio in \cref{eq:kg-ratio}.
Prior literature on $K_G$ has identified a family of interesting instances which we call \emph{Hermite projection games}.
The ``matrix'' in a Hermite projection game is actually a linear operator on the space
of functions $\R^n \to \R$
(which can be discretized into a finite-dimensional matrix at the very end).
A Hermite projection game is defined to be
a linear combination of the operators $\bm{\Pi}_d$ where $\bm{\Pi}_d$ projects a function to the degree-$d$ part of its Hermite polynomial expansion.
Note that we take the limit $n \to \infty$ in a Hermite projection game,
but the coefficients on the matrices $\bm{\Pi}_d$ are constant with respect to $n$.

For example, $\bm{\Pi}_1$ achieves value $\frac{\pi}{2}$ which is known to be the largest possible value of \cref{eq:kg-ratio} among PSD matrices, the so-called ``little Grothendieck inequality''.
The Davie--Reeds operator which achieves the best lower bound on $K_G$ to date is a Hermite projection game,
\[
    A_{DR} = \bm{\Pi}_1 - \lam^* \bI, \qquad \lam^* \approx 0.19748\,.
\]
A remarkable result of Raghavendra and Steurer shows that Hermite projection games are complete for $K_G$: for every $\eps > 0$, there exists a Hermite projection game such that \cref{eq:kg-ratio} is at least $K_G - \eps$, regardless of the true value of $K_G$ \cite{raghavendra2009towards}.

Our improvement to the Davie--Reeds construction is also a Hermite projection game, specifically
\[
    A = \bm{\Pi}_1 - \lam^* \bI - \eps \bm{\Pi}_3
\]
for a small constant $\eps > 0$.
Let us give some intuition for why we choose this perturbation of the Davie--Reeds operator using the viewpoint of $K_G$ through nonlocal games.

The word ``game'' refers to the interpretation of \cref{eq:kg-ratio} as the gap between the optimal quantum and classical strategies for a 2-player XOR game.
This type of game is parameterized by a matrix $A \in \R^{n \times n}$.
There are two players, Alice and Bob, who separately receive an index $i_A, i_B \in [n]$ such that the probability of receiving pair $(i_A, i_B)$ is proportional to $|A_{i_Ai_B}|$.\footnote{We can assume without loss of generality that the matrix $A$ is normalized so that $\sum_{i,j =1}^n |A_{ij}| =1$, since one can easily see from \cref{eq:kg-ratio} that $K_G$ is scale invariant.} 
Their goal is to output bits $a,b \in \{\pm1\}$ such that $a b = \sign(A_{i_Ai_B})$.

If no communication is allowed between Alice and Bob, their optimal strategy has bias (i.e., probability of winning minus probability of losing) exactly equal to the denominator of \cref{eq:kg-ratio}.
When Alice and Bob have access to a shared quantum state, instead the bias of their optimal strategy is exactly the numerator of \cref{eq:kg-ratio}.
Therefore, $K_G$ is equal to the maximum quantum-versus-classical advantage across all possible 2-player XOR games.
The famous \emph{CHSH game} is a simple example with $n=2$ showing $K_G > 1$ (there exists a quantum strategy which is strictly better than all possible classical strategies; this is known as a ``Bell inequality'').

Consider three different 2-player XOR games which are Hermite projection games:
\begin{enumerate}
    \item Alice and Bob receive $n$-dimensional vectors $X,Y$ with probability proportional to $e^{-\norm{X}_2^2/2 - \norm{Y}_2^2/2}|\ip{X,Y}|$. Their goal is to output $ab = \sign(\ip{X,Y})$.
    \item Alice and Bob both receive the same $n$-dimensional Gaussian vector $X$. Their goal is to output $ab = -1$.
    \item Alice and Bob play game 1 with probability $1-p$ or game 2 with probability $p$.
\end{enumerate}

The operator corresponding to the first game is $\bm{\Pi}_1$. The operator corresponding to the second game is $-\bI$. The operator corresponding to the third game is the (rescaled) Davie--Reeds operator $(1-p)\bm{\Pi}_1 - p\bI$.

The optimal classical strategy for the first game is to output $a = \sign(X_1)$ and $b = \sign(Y_1)$. The second game is trivially winnable with probability 1. The third game, however, turns out to be strictly harder than game 1 (which is good since this leads to a larger lower bound for $K_G$).
At the intuitive level, this is because we have confused Alice and Bob by ``negating'' the correct answer with probability $p$: when $X$ and $Y$ are correlated vectors,
we could be in game 1, in which case the goal of Alice and Bob is to answer $ab = 1$,
or we could be in game 2, in which case Alice and Bob are supposed to answer $ab = -1$ instead.

Extending this intuition,
we should confuse Alice and Bob even more by introducing additional alternations.
For example, a candidate harder game would be $ab = \sign\left(f\left(\frac{\ip{X,Y}}{\norm{X}\norm{Y}}\right)\right)$ for a function $f: [-1,1] \to \R$ that oscillates,
since Alice and Bob would need to know the angle between $X$ and $Y$ very precisely in order to win this game.
The work of König \cite{konig2001extremal}
suggests a concrete instance of this type, although its analysis remains an open problem.
The candidate instance is an operator on functions $\R^n \to \R$ whose ``entries'' are (more precisely, the kernel function of the operator is)
\[
    A_G(X,Y) = e^{-\norm{X}_2^2/2 - \norm{Y}_2^2/2}\sin \ip{X, Y}\,.
\]

$A_G$ is similar to a Hermite projection game but it is not quite one.
If we write the Hermite polynomial expansion of $\sin \ip{X,Y}$
\[
    \sin \ip{X,Y} = \sum_{\al, \beta \in \N^n} c_{\al\beta} \He_\al(X)\He_\beta(Y)
\]
then, in $A_G$, the terms with $\alpha = \beta$ are a Hermite projection game,
while the terms with $\alpha \neq \beta$ are additional ``off-diagonal matrices'' with respect to the Hermite polynomial decomposition.

It can be shown that the Hermite coefficients $c_{\al\al}$ are zero unless $|\al|$ is odd
and they alternate sign for $|\al| = 1,3,5,7,\dots$
This motivates us to consider Hermite projection games of the form $\bm{\Pi}_1 - c_3 \bm{\Pi}_3 + c_5 \bm{\Pi}_5 + \cdots$.
We propose to bridge the gap from the instances $\bm{\Pi}_1$ and $\bm{\Pi}_1 - \lam^* \bI$ to a candidate hard instance such as $A_G$
by steadily adding more low-degree Hermite projectors $\bm{\Pi}_k$ with alternating coefficients.
While the quantitative improvement in \cref{thm:new-lb} is small, the proof of \cref{thm:new-lb} confirms the intuition behind this approach, showing that an additional alternation $\bm{\Pi}_3$ indeed makes the Davie--Reeds game harder.

\section{Preliminaries}

We start with some preliminaries on Hermite projection games.

\newcommand{\SDP}{\mathrm{sdp}}
\let\sdp\SDP

For a measure space $(\Om, \mu)$ and a bounded linear operator $A : L^\infty(\mu) \to L^1(\mu)$,
define the quantities:
\begin{align*}
\val(A) &= \sup_{f, g : \Om \to \{\pm 1\}} \int_\Om (Af)(X) g(X)\, d\mu(X) \\
\SDP(A) &= \sup_{\calH} \sup_{\substack{f,g: \Om\to \mathcal H\\ \|f(X)\|_{\mathcal H}= 1, \|g(X)\|_\calH = 1}} \int_\Om \langle(Af)(X), g(X)\rangle_{\mathcal H} \, d\mu(X)
\end{align*}
where the supremum over $\calH$ is over real Hilbert spaces, and we lift $A$ to an operator on $f : \Om \to \calH$ by applying it separately to each coordinate of $\calH$.
More precisely, select an orthonormal basis $v_i$ for $\calH$, represent $f(X) = \sum_{i} f_i(X) v_i$, and define $Af = \sum_{i} (Af_i)(X) v_i$.
The case where $\mu$ is the discrete counting measure on $\{1,2,\dots, n\}$
corresponds to finite-dimensional matrices $A \in \R^{n \times n}$.

The boundedness of the operator $A: L^\infty(\mu) \to L^1(\mu)$ is exactly equivalent to $\val(A) < \infty$.
On the other hand, while $\sdp(A)$ could ostensibly be infinite, the Grothendieck inequality states that $\sdp(A)$ is in fact finite
and only a constant factor larger than $\val(A)$.

\begin{theorem}[{\cite[Theorems 2.4 and 2.5]{pisier2012grothendieck}}]
\label{lem:grothendieck-operator}
    Let $(\Om, \mu)$ be a measure space and let $A : L^\infty(\mu) \to L^1(\mu)$ be a bounded linear operator.
    Then there is a constant $K$ such that
    \[
        \sdp(A) \leq K \val(A)
    \]
    Furthermore, the best constant $K$ is equal to $K_G$ defined in \cref{eq:kg-ratio}.
\end{theorem}

Specializing to the case of Gaussian measure,
let $n \in \N$ and let $\gam = \gam^{(n)}$ be the $n$-dimensional standard Gaussian measure.
The space $L^2(\gam)$ has the Gaussian inner product $\ip{f,g}_\gam = \E_{X \sim \gam} f(X)g(X)$ and the corresponding norm $\norm{f}_2 = \left(\E_{X \sim \gam} f(X)^2\right)^{1/2}$.
For a multi-index $\al \in \N^n$, let $\He_\al(X)$ be the $\al$ Hermite polynomial (probabilists' convention)
which together form an orthogonal basis for $L^2(\gam)$.
For $\al \in \N^n$ define the notation $\al! = \prod_{i = 1}^n \al_i!$ and $|\al| = \sum_{i=1}^n \al_i$.

Let $\bm{\Pi}_k:L^2(\gam)\to \mathrm{span}\{\He_\alpha:|\alpha|=k\}$ be the projection operator to the degree-$k$ Hermite polynomials.
Our candidate integrality gap instances for the Grothendieck constant are
\emph{Hermite projection games} of the form
\[
    A = \sum_{k = 0}^\infty c_k \bm{\Pi}_k
\]
for $\sup_{k \in \N} |c_k| < \infty$.
The SDP value of a Hermite projection game has a simple formula:
it is just equal to the spectral norm of $A$.

\begin{proposition}\label{lem:sdp-value}
    Let $A = \sum_{k = 0}^\infty c_k \bm{\Pi}_k$ be a Hermite projection game. Then
    \[\lim_{n \to \infty} \SDP(A) = \sup_{k \in \N} |c_k|\,.\]
\end{proposition}
\begin{proof}
We first prove the upper bound.
Let $L^2(\gam; \calH)$ denote the space of functions $f : \R^n \to \calH$ such
that $\E_{X \sim \gam} \norm{f(X)}_\calH^2 < \infty$.
The degree-\(k\) Hermite subspaces are orthogonal in this space, so
\[
\|Af\|_{L^2(\gamma;\mathcal H)}^2
=
\sum_{k\ge 0} c_k^2 \|\bm{\Pi}_kf\|_{L^2(\gamma;\mathcal H)}^2
\le
\left(\sup_{k\ge 0}|c_k|\right)^2\sum_{k\ge 0}\|\bm{\Pi}_k f\|_{L^2(\gamma;\mathcal H)}^2
=
\left(\sup_{k\ge 0}|c_k|\right)^2\|f\|_{L^2(\gamma;\mathcal H)}^2.
\]
For any $f: \R^n \to \calH$ such that \(\|f(X)\|_{\mathcal H}\le 1\), we have \(\|f\|_{L^2(\gamma;\mathcal H)}\le 1\). Therefore
\[
\E_{X \sim \gamma}\left\|Af(X)\right\|_{\mathcal H}
\le
\|Af\|_{L^2(\gamma;\mathcal H)}
\le \sup_{k\in \N}|c_k|.
\]
This gives an upper bound for $\sdp(A)$ since it holds that
\[\SDP(A) = \sup_{\calH} \sup_{\substack{f,g: \Om\to \mathcal H\\ \|f(X)\|_{\mathcal H}\le 1, \|g(X)\|_\calH \le 1}} \E_{X \sim \gam} \langle(Af)(X), g(X)\rangle_{\mathcal H} = \sup_{\calH} \sup_{\substack{f: \Om\to \mathcal H\\ \|f(X)\|_{\mathcal H}\le 1}} \E_{X \sim \gam} \norm{(Af)(X)}_{\mathcal H}\,.\]

For the lower bound, fix some constant $k$. For \(n\ge k\), let \(\mathcal H\) be the Hilbert space with orthonormal basis
\(\{e_S:S\subseteq[n],\,|S|=k\}\), and define
\[
\Psi(x):=\binom{n}{k}^{-1/2}\sum_{|S|=k}\left(\prod_{i\in S}x_i\right)e_S.
\]
Notice that 
\[
\norm{\Psi}_{L^2(\gamma; \mathcal{H})}^2 = \E_{X\sim \gamma} \|\Psi(X)\|_{\mathcal H}^2 = 
\E_{X\sim \gamma} \binom{n}{k}^{-1}\sum_{|S|=k}\prod_{i\in S}X_i^2
=
\binom{n}{k}^{-1}\sum_{|S|=k}\prod_{i\in S}\E X_i^2
= 1
\]
by independence. Since $\norm{\Psi(X)}_{\mathcal H}^2$ is the average of $\binom{n}{k}$ monomials and changing one coordinate affects only $k/n = O_k(1/n)$ fraction of them, the Efron–Stein inequality yields \(\Var(\norm{\Psi(X)}_{\mathcal H}^2)=O_k(1/n)\). Thus
\begin{equation} \label{eq:converge}
\E_{X\sim \gamma}\left(\norm{\Psi(X)}_{\mathcal H}-1\right)^2\le  \E_{X\sim \gamma}\left(\norm{\Psi(X)}_{\mathcal H}^2-1\right)^2\to 0.
\end{equation}
That is to say, $\norm{\Psi(X)}_\calH$ converges to 1 in $L^2$.
Now define
\[
f(x):=
\begin{cases}
\frac{\Psi(x)}{\norm{\Psi(x)}_{\mathcal H}} & \norm{\Psi(x)}_{\mathcal H}\neq 0,\\[1mm]
0 & \norm{\Psi(x)}_{\mathcal H}=0.
\end{cases}
\]
Then $f$ is a feasible solution for the SDP. Moreover, $\norm{f-\Psi}_{L^2(\gamma;\mathcal H)}^2 = \E_{X\sim \gamma}(\norm{\Psi(X)}_{\mathcal H}-1)^2$ converges to $0$ by \cref{eq:converge}, which implies that the Hermite weights of $f$ converge to those of $\Psi$:
\begin{equation}\label{eq:important}
    \|\bm{\Pi}_kf\|_{L^2(\gamma;\mathcal H)}^2 \to \norm{\bm{\Pi}_k\Psi}_{L^2(\gamma;\calH)}^2 = \norm{\Psi}_{L^2(\gamma;\calH)}^2 = 1
\qquad\text{and} \qquad
\sum_{j\neq k}\|\bm{\Pi}_jf\|_{L^2(\gamma;\mathcal H)}^2 \to 0.
\end{equation}
By Cauchy--Schwarz and the pointwise bound \(\|f(x)\|_{\mathcal H}\le 1\), we have
\[
\operatorname{sign}(c_k)\,
\langle Af(X),f(X)\rangle_{\mathcal{H}}
\le
\|Af(X)\|_{\mathcal H}.
\]
Taking expectations, it follows that
\begin{align*}
\E_{X\sim \gamma}\|Af(X)\|_{\mathcal H}
&\ge
\operatorname{sign}(c_k)\,
\E_{X\sim \gamma}\left\langle Af(X),f(X)\right\rangle_{\mathcal{H}} \\
&=
\operatorname{sign}(c_k)\sum_{j=0}^\infty c_j\,\|\bm{\Pi}_jf\|_{L^2(\gamma;\mathcal H)}^2 \\
&\ge
|c_k|\cdot\left\|\bm{\Pi}_kf\right\|_{L^2(\gamma;\mathcal H)}^2
-
\left(\sup_{\kappa\geq 0}|c_\kappa|\right)\sum_{j\neq k}\left\|\bm{\Pi}_jf\right\|_{L^2(\gamma;\mathcal H)}^2.
\end{align*}
By \Cref{eq:important}, the above expression converges to $|c_k|$. Since $f$ is feasible, this is also a lower bound for the SDP value. Since this holds for all $k$, this completes the proof.
\end{proof}

\begin{lemma}\label{lem:rotation-invariant}
    Let $A$ be a Hermite projection game and let $\val_A(f,g) = \ip{Af, g}_\gam = \E_{X \sim \gam} (Af)(X)g(X)$.
    Then $\val_A(f,g) = \val_A(f \circ T, g \circ T)$ for all orthogonal matrices $T \in O(n)$ and $f,g: \R^n \to \R$.
\end{lemma}
\begin{proof}
We claim that this property holds more generally for operators $A$ which are rotation-equivariant, satisfying $A(f \circ T) = (Af) \circ T$ for all $f:\R^n \to \R$ and $T \in O(n)$.
To prove this,
\begin{align*}
\val_A(f\circ T,g\circ T)
&= \E_{X\sim\gamma} (A(f \circ T))(X)\,g(TX)\\
&= \E_{X\sim\gamma} (Af)(TX)\,g(TX) & (\text{Rotation invariance of }A)\\
&= \E_{X'\sim\gamma} (Af)(X')\,g(X') = \val_A(f,g) & (\text{Rotation invariance of $\gam$})
\end{align*}

To use this for Hermite projection games, we note that each of the operators $\bm{\Pi}_k$ is rotation-equivariant.
One way to prove this is using that the Ornstein--Uhlenbeck noise operator $U_t$ decomposes as $U_t = \sum_{k=0}^\infty e^{-kt} \bm{\Pi}_k$
and the Ornstein--Uhlenbeck operator is rotation-equivariant \cite[Section 1.4]{bogachev1998gaussian}.
Let $R_T$ be the rotation operator $f \mapsto f \circ T$.
Since $R_T$ commutes with $U_t$, it preserves each eigenspace of $U_t$.
Therefore, it preserves the space of degree-$k$ Hermite polynomials, $R_T\bm{\Pi}_k = \bm{\Pi}_kR_T$.
\end{proof}

\begin{lemma}\label{lem:value-stability}
    Let $A$ be a Hermite projection game and $f,\widetilde f, g : \R^n \to \{\pm1\}$.
    Then \[\abs{\val_A(f,g) - \val_A(\widetilde f, g)} \leq \norm{A} \cdot \|f - \widetilde f\|_2.\]
\end{lemma}
\begin{proof}
    By Cauchy--Schwarz,
    \[
    \abs{\E_{X \sim \gam} ((Af)(X)- (A\widetilde f)(X))g(X)}
\le \norm{g}_2\,\norm*{A(f-\widetilde f)}_2
\le \norm{A}\cdot\norm*{f-\widetilde f}_2
    \]
    since $\norm{g}_2 = 1$ and by definition of the operator norm.
\end{proof}

\section{Davie--Reeds game}

The Davie--Reeds game is $A_{DR} = \bm{\Pi}_1 - \lam^* \bI$ where $\lambda^* \approx 0.19748$. The constant $\lambda^*$ is chosen as $\lambda^* = 2C^* \phi(C^*)$ where $C^* \approx 0.25573$ is defined as follows.

\begin{lemma}\label{lem:unique-c}
    There is a unique solution to $4\phi(C)^2 - 4\Phi(-C) + 1 = 0$ for $C \in (0,1)$ which is $C^* \approx 0.25573$
    where $\phi(x) = \frac{1}{\sqrt{2\pi}}e^{-x^2/2}$ and $\Phi(x) = \int_{-\infty}^x \phi(t) dt$.
\end{lemma}
\begin{proof}
    Let $H(C) := 4\phi(C)^2 - 4\Phi(-C) + 1$. Uniqueness of the root follows from the computations
    \[
        H(0) = \frac{2}{\pi} - 1 < 0\,, \qquad H(1) \approx 0.59958 > 0\,, \qquad H' = 4\phi (1 - 2 C \phi) > 0 \text{ for }C \in [0,1]\,. \qedhere
    \]
\end{proof}

\begin{lemma}\label{lem:dr-optimizers}
    Let $\widetilde f, \widetilde g : \R^n \to \{\pm 1\}$. Then $\widetilde f, \widetilde g$ are Davie--Reeds optimal if and only if there exists $T \in O(n)$ such that, letting $f = \widetilde f \circ T$ and $g = \widetilde g \circ T$:
\begin{enumerate}[(i)]
    \item Up to a set of measure zero,
    \[
        f(X) = \begin{cases}
            1 & X_1 \geq C^*\\
            - g(X) & -C^* < X_1 < C^*\\
            -1 & X_1 \leq -C^*
        \end{cases}
        \qquad g(X) = \begin{cases}
            1 & X_1 \geq C^*\\
            - f(X) & -C^* < X_1 < C^*\\
            -1 & X_1 \leq -C^*
        \end{cases}
    \]
    \item $\bm{\Pi}_1h = 0$ where $h : \R^n \to \{-1,0,1\}$ is the restriction of $f$ to the set $\{x \in \R^n : |x_1| \leq C^*\}$.
\end{enumerate}
\end{lemma}

There are many functions $f,g$ with these two properties, see \Cref{fig:squarewave}.
We will prove \cref{lem:dr-optimizers} in the next section.
For now, we prove that all such functions have a $\bm{\Pi}_3$ component which is $\Omega(1)$.

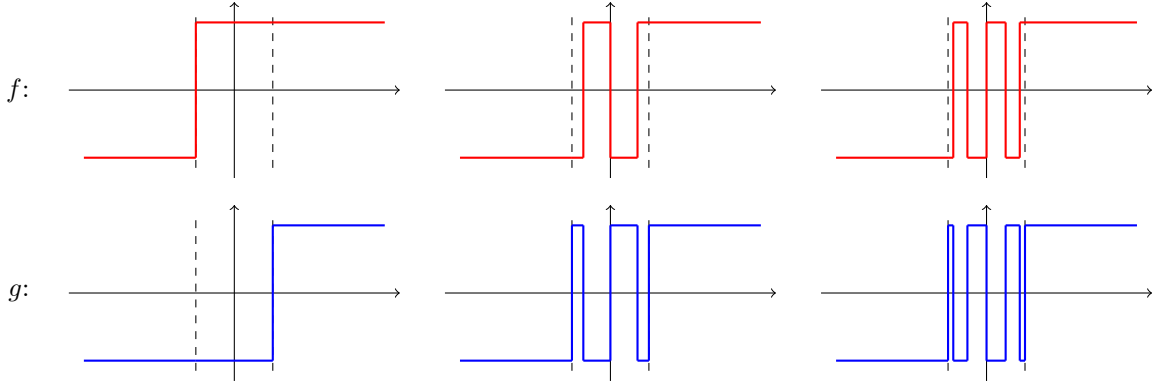
\begin{figure}[th]
\centering
\begin{tikzpicture}[x=2cm,y=0.9cm] 
    \def\C{0.25573}
    \def\Amid{0.18009}
    \def\Amore{0.12708}
    \def\Bmore{0.22101}


    \begin{scope}[shift={(0,0)}]
        \draw[->] (-1.1,0) -- (1.1,0);
        \draw[->] (0,-1.3) -- (0,1.3);

        \draw[dashed] (-\C,-1.15) -- (-\C,1.15);
        \draw[dashed] (\C,-1.15) -- (\C,1.15);

        \draw[red, thick] (-1,-1) -- (-\C,-1);
        \draw[red, thick] (-\C,-1) -- (-\C,1);
        \draw[red, thick] (-\C,1) -- (\C,1);
        \draw[red, thick] (\C,1) -- (1,1);
    \end{scope}

    \begin{scope}[shift={(5cm,0)}]
        \draw[->] (-1.1,0) -- (1.1,0);
        \draw[->] (0,-1.3) -- (0,1.3);

        \draw[dashed] (-\C,-1.15) -- (-\C,1.15);
        \draw[dashed] (\C,-1.15) -- (\C,1.15);


        \draw[red, thick] (-1,-1) -- (-\C,-1);
        \draw[red, thick] (-\C,-1) -- (-\Amid,-1);
        \draw[red, thick] (-\Amid,-1) -- (-\Amid,1);
        \draw[red, thick] (-\Amid,1) -- (0,1);
        \draw[red, thick] (0,1) -- (0,-1);
        \draw[red, thick] (0,-1) -- (\Amid,-1);
        \draw[red, thick] (\Amid,-1) -- (\Amid,1);
        \draw[red, thick] (\Amid,1) -- (\C,1);
        \draw[red, thick] (\C,1) -- (1,1);
    \end{scope}

    \begin{scope}[shift={(10cm,0)}]
        \draw[->] (-1.1,0) -- (1.1,0);
        \draw[->] (0,-1.3) -- (0,1.3);

        \draw[dashed] (-\C,-1.15) -- (-\C,1.15);
        \draw[dashed] (\C,-1.15) -- (\C,1.15);


        \draw[red, thick] (-1,-1) -- (-\C,-1);
        \draw[red, thick] (-\C,-1) -- (-\Bmore,-1);
        \draw[red, thick] (-\Bmore,-1) -- (-\Bmore,1);
        \draw[red, thick] (-\Bmore,1) -- (-\Amore,1);
        \draw[red, thick] (-\Amore,1) -- (-\Amore,-1);
        \draw[red, thick] (-\Amore,-1) -- (0,-1);
        \draw[red, thick] (0,-1) -- (0,1);
        \draw[red, thick] (0,1) -- (\Amore,1);
        \draw[red, thick] (\Amore,1) -- (\Amore,-1);
        \draw[red, thick] (\Amore,-1) -- (\Bmore,-1);
        \draw[red, thick] (\Bmore,-1) -- (\Bmore,1);
        \draw[red, thick] (\Bmore,1) -- (\C,1);
        \draw[red, thick] (\C,1) -- (1,1);
    \end{scope}

    \node[anchor=east] at (-1.3,0) {$f$:};


    \begin{scope}[shift={(0,-3)}]
        \draw[->] (-1.1,0) -- (1.1,0);
        \draw[->] (0,-1.3) -- (0,1.3);

        \draw[dashed] (-\C,-1.15) -- (-\C,1.15);
        \draw[dashed] (\C,-1.15) -- (\C,1.15);

        \draw[blue, thick] (-1,-1) -- (-\C,-1);
        \draw[blue, thick] (-\C,-1) -- (\C,-1);
        \draw[blue, thick] (\C,-1) -- (\C,1);
        \draw[blue, thick] (\C,1) -- (1,1);
    \end{scope}

    \begin{scope}[shift={(5cm,-3)}]
        \draw[->] (-1.1,0) -- (1.1,0);
        \draw[->] (0,-1.3) -- (0,1.3);

        \draw[dashed] (-\C,-1.15) -- (-\C,1.15);
        \draw[dashed] (\C,-1.15) -- (\C,1.15);


        \draw[blue, thick] (-1,-1) -- (-\C,-1);
        \draw[blue, thick] (-\C,-1) -- (-\C,1);
        \draw[blue, thick] (-\C,1) -- (-\Amid,1);
        \draw[blue, thick] (-\Amid,1) -- (-\Amid,-1);
        \draw[blue, thick] (-\Amid,-1) -- (0,-1);
        \draw[blue, thick] (0,-1) -- (0,1);
        \draw[blue, thick] (0,1) -- (\Amid,1);
        \draw[blue, thick] (\Amid,1) -- (\Amid,-1);
        \draw[blue, thick] (\Amid,-1) -- (\C,-1);
        \draw[blue, thick] (\C,-1) -- (\C,1);
        \draw[blue, thick] (\C,1) -- (1,1);
    \end{scope}

    \begin{scope}[shift={(10cm,-3)}]
        \draw[->] (-1.1,0) -- (1.1,0);
        \draw[->] (0,-1.3) -- (0,1.3);

        \draw[dashed] (-\C,-1.15) -- (-\C,1.15);
        \draw[dashed] (\C,-1.15) -- (\C,1.15);


        \draw[blue, thick] (-1,-1) -- (-\C,-1);
        \draw[blue, thick] (-\C,-1) -- (-\C,1);
        \draw[blue, thick] (-\C,1) -- (-\Bmore,1);
        \draw[blue, thick] (-\Bmore,1) -- (-\Bmore,-1);
        \draw[blue, thick] (-\Bmore,-1) -- (-\Amore,-1);
        \draw[blue, thick] (-\Amore,-1) -- (-\Amore,1);
        \draw[blue, thick] (-\Amore,1) -- (0,1);
        \draw[blue, thick] (0,1) -- (0,-1);
        \draw[blue, thick] (0,-1) -- (\Amore,-1);
        \draw[blue, thick] (\Amore,-1) -- (\Amore,1);
        \draw[blue, thick] (\Amore,1) -- (\Bmore,1);
        \draw[blue, thick] (\Bmore,1) -- (\Bmore,-1);
        \draw[blue, thick] (\Bmore,-1) -- (\C,-1);
        \draw[blue, thick] (\C,-1) -- (\C,1);
        \draw[blue, thick] (\C,1) -- (1,1);
    \end{scope}

    \node[anchor=east] at (-1.3,-3) {$g$:};

\end{tikzpicture}
\caption{One-dimensional ``square wave'' examples of Davie--Reeds optimizers. In each example, $f=g=\sign(x)$ outside the strip $[-C^*,C^*]$, while $g=-f$ inside the strip, and the breakpoints are chosen so that $\E_{x \sim \calN(0,1)} xf(x) \bm{1}_{|x| \leq C^*} = 0$.}
\label{fig:squarewave}
\end{figure}

We say that $f,g$ are \emph{Davie--Reeds strip functions}
if there exists $T \in O(n)$
such that $f \circ T$ and $g \circ T$ satisfy item (i) in \cref{lem:dr-optimizers}.

\begin{lemma}\label{lem:he3-lower-bound}
    Suppose that $f, g : \R^n \to \{\pm1\}$ are Davie--Reeds strip functions.
    Then
    \[
        \E_{X \sim \gam} (\bm{\Pi}_3 f)(X) (\bm{\Pi}_3 g)(X) \geq 0.046.
    \]
\end{lemma}

\begin{proof}
Applying an orthogonal change of coordinates, which does not change $\E_{X \sim \gam}(\bm{\Pi}_3 f)(X)(\bm{\Pi}_3 g)(X)$ by \cref{lem:rotation-invariant}, we may assume that
$f,g$ are strip functions in the direction $X_1$.
Define
\begin{align*}
u(X)&:=\sign(X_1) \bm{1}_{|X_1| \geq C^*}
\qquad
h(X):=f(X)\bm 1_{|X_1| \leq C^*}\\
f(X)&=u(X)+h(X) \qquad g(X)=u(X)-h(X).
\end{align*}
Applying $\bm{\Pi}_3$, we have
\begin{align*}
\E_{X \sim \gam} (\bm{\Pi}_3 f)(X)(\bm{\Pi}_3 g)(X)
&= \E_{X \sim \gam} ((\bm{\Pi}_3 u)(X) + (\bm{\Pi}_3 h)(X)) ((\bm{\Pi}_3 u)(X) - (\bm{\Pi}_3 h)(X)) \\
&= \E_{X \sim \gam} ((\bm{\Pi}_3 u)(X))^2 - ((\bm{\Pi}_3 h)(X))^2 = 
\|\bm{\Pi}_3 u\|_2^2-\|\bm{\Pi}_3 h\|_2^2. \numberthis \label{eq:he3-split}
\end{align*}
We first compute $\|\bm{\Pi}_3 u\|_2^2$.
Since $u$ depends only on $X_1$, its degree-$3$ Hermite projection is a multiple of $\He_3(X_1)$ only. Thus
\[
\bm{\Pi}_3 u = \widehat u(3,0,\dots,0)\He_3(X_1)
\]
and consequently
\[
\|\bm{\Pi}_3 u\|_2^2 = 6\,\widehat u(3,0,\dots,0)^2.
\]
Expanding, we obtain
\[
6\widehat u(3,0,\dots, 0) = \E_{X\sim\gam} u(X)\He_3(X_1)
=
2\int_{C^*}^{\infty}(x^3-3x) \phi(x)\,dx\,.
\]
For $C^*\approx 0.25573$, this integral numerically evaluates to
\begin{equation}\label{eq:u-lb}
\|\bm{\Pi}_3 u\|_2^2 \approx 0.0868.
\end{equation}
Next we bound $\|\bm{\Pi}_3 h\|_2^2$ from above.
If $\bm{\Pi}_3 h=0$ there is nothing to prove, so assume otherwise.  By Cauchy--Schwarz,
\begin{equation}\label{eq:strip-cs}
\left\|\bm{\Pi}_3 h\right\|_2^2
\le
\|h\|_2^2\,\left\|\bm 1_{|X_1|<C^*} \frac{\bm{\Pi}_3 h}{\|\bm{\Pi}_3 h\|_2} \right\|_2^2
=
\|h\|_2^2\,\left(\E \bm 1_{|X_1|<C^*} \left(\frac{\bm{\Pi}_3 h}{\|\bm{\Pi}_3 h\|_2}\right)^2 \right).
\end{equation}
We will show that
\begin{equation}\label{eq:strip-bound}
\E\bm 1_{|X_1|<C^*} \left(\frac{\bm{\Pi}_3 h}{\|\bm{\Pi}_3 h\|_2}\right)^2\le \E \bm{1}_{|X_1|<C^*}
\end{equation}
Towards \cref{eq:strip-bound}, write the Hermite expansion
\[
\frac{\bm{\Pi}_3 h}{\|\bm{\Pi}_3 h\|_2}(X)=\sum_{i=0}^3 \He_i(X_1)\,r_i(X_2,\dots,X_n),
\]
where each $r_i$ lies in the degree-$(3-i)$ Hermite subspace in the remaining variables. Since $\bm 1_{|X_1|<C^*}$ depends only on $X_1$, orthogonality in $(X_2,\dots,X_n)$ gives
\[
\E\bm 1_{|X_1|<C^*} \left(\frac{\bm{\Pi}_3 h}{\|\bm{\Pi}_3 h\|_2}\right)^2
=
\sum_{i=0}^3 \E\left(\bm 1_{|X_1|<C^*} \He_i(X_1)^2\right)\|r_i\|_2^2 \le
\max_{0\le i\le 3}\frac{\E\bm 1_{|X_1|<C^*} \He_i(X_1)^2}{i!}
\]
where the inequality holds because
\[
1=\left\|\frac{\bm{\Pi}_3 h}{\|\bm{\Pi}_3 h\|_2}\right\|_2^2
=\sum_{i=0}^3 i!\,\|r_i\|_2^2.
\]
Notice that on $|x|\leq C^*<1$,
\[
\frac{\He_0(x)^2}{0!}=1,\qquad
\frac{\He_1(x)^2}{1!}=x^2\le1,\qquad
\frac{\He_2(x)^2}{2!}=\frac{(x^2-1)^2}{2}\le1,\qquad
\frac{\He_3(x)^2}{3!}=\frac{(x^3-3x)^2}{6}<1.
\]
Therefore
\[
\frac{\E\bm 1_{|X_1|<C^*} \He_i(X_1)^2}{i!}
\le \E\bm 1_{|X_1|<C^*}
\qquad (0\le i\le 3).
\]
This proves \cref{eq:strip-bound}.
On the other hand, since $h = \pm 1$ on the strip and $h = 0$ otherwise,
\[\|h\|_2^2 = \E \bm{1}_{|X_1|<C^*} = 2\Phi(C^*) - 1 \approx 0.20184\]
Plugging this into \cref{eq:strip-cs} shows that,
\begin{equation}\label{eq:h-ub}
\norm{\bm{\Pi}_3h}_2^2 \lessapprox (0.20184)^2 \approx 0.0408.
\end{equation}
Combining \cref{eq:he3-split,eq:u-lb,eq:h-ub},
\[
\E (\bm{\Pi}_3 f)(X)(\bm{\Pi}_3 g)(X)
\ge
0.0868-0.0408
\ge
0.046
\]
This proves the lemma.
\end{proof}

\section{Perturbing the Davie--Reeds game}

For $\eps, \lam \in \R$ define the \emph{perturbed Davie--Reeds game} by
\[
    A_{\eps, \lam} = \bm{\Pi}_1 - \lam \bI - \eps \bm{\Pi}_3\,.
\]

The Davie--Reeds game itself uses $\eps = 0$ and $\lam = \lam^*$.
We fix $\lam = \lam^*$ throughout and let $A_{\eps} = A_{\eps, \lam^*}$.

\begin{theorem}[Improvement over Davie--Reeds]\label{thm:perturbed-dr}
    $\val(A_\eps) \leq \val(A_{DR})-\eps(0.046-12(2\eps)^{1/4})$.
\end{theorem}

The main lemma for the proof is a stability estimate for the optimizers
of the Davie--Reeds game.
Let $\val_A(f,g) = \E_{X \sim \gam} (Af)(X)g(X)$.

\begin{lemma}
\label{lem:dr-robust}
    Assume $f, g : \R^n \to \{\pm 1\}$ have $\val_{A_{DR}}(f, g) \geq \val(A_{DR}) - \eta$.
    Then there exist Davie--Reeds strip functions $f_{DR}, g_{DR} : \R^n \to \{\pm 1\}$ such that $\norm{f - f_{DR}}_2 \leq 6\eta^{1/4}$ and $\norm{g - g_{DR}}_2 \leq 6\eta^{1/4}$.
\end{lemma}
We can quickly deduce the theorem from the lemma.

\begin{proof}[Proof of \cref{thm:perturbed-dr} from \cref{lem:dr-robust}.]
    Let $f, g : \R^n \to \{\pm 1\}$.
    The value achieved by $f$ and $g$ is
    \begin{align}
        \val_{A_\eps}(f,g)
        &= \E_{X \sim \gam} (\bm{\Pi}_1 f)(X) (\bm{\Pi}_1 g)(X)
           - \lam^* \E_{X \sim \gam} f(X)g(X)
           - \eps \E_{X \sim \gam} (\bm{\Pi}_3 f)(X) (\bm{\Pi}_3 g)(X) \notag\\
        &= \val_{A_{DR}}(f,g) - \eps \E_{X \sim \gam} (\bm{\Pi}_3 f)(X) (\bm{\Pi}_3 g)(X).
        \label{eq:perturbed-game-value}
    \end{align}
    We split into two cases. In the first case
        $\val_{A_{DR}}(f,g) \ge \val(A_{DR}) - 2\eps$.
    By \cref{lem:dr-robust}, there exist Davie--Reeds strip functions
    $f_{DR},g_{DR}$ such that
    \[
        \norm{f-f_{DR}}_2 \le 6(2\eps)^{1/4},
        \qquad
        \norm{g-g_{DR}}_2 \le 6(2\eps)^{1/4}.
    \]
    Using \cref{lem:value-stability} with $A=\bm{\Pi}_3$ twice,
    \begin{align*}
        \abs{\E (\bm{\Pi}_3 f)(X)(\bm{\Pi}_3 g)(X)-\E (\bm{\Pi}_3 f_{DR})(X)(\bm{\Pi}_3 g_{DR})(X)}\le \norm{f-f_{DR}}_2+\norm{g-g_{DR}}_2
        \le 12(2\eps)^{1/4}.
    \end{align*}
    Hence by \cref{lem:he3-lower-bound},
    \[
        \E (\bm{\Pi}_3 f)(X)(\bm{\Pi}_3 g)(X)
        \ge 0.046 - 12(2\eps)^{1/4}.
    \]
    Plugging this into \cref{eq:perturbed-game-value},
    \[
        \val_{A_\eps}(f,g)
        \le \val(A_{DR}) - \eps\left(0.046-12(2\eps)^{1/4}\right).
    \]
    In the second case $\val_{A_{DR}}(f,g) < \val(A_{DR}) - 2\eps$.
    By Cauchy--Schwarz and since $\bm{\Pi}_3$ is an orthogonal projection,
    \[
        \abs{\E_{X \sim \gam} (\bm{\Pi}_3 f)(X) (\bm{\Pi}_3 g)(X)}
        \le \norm{\bm{\Pi}_3 f}_2 \norm{\bm{\Pi}_3 g}_2
        \le \norm{f}_2 \norm{g}_2 = 1.
    \]
    Therefore
    \[
        \val_{A_\eps}(f,g)
        \le \val(A_{DR}) - 2\eps + \eps
        = \val(A_{DR}) - \eps.
    \]
Combining the two cases yields the desired theorem.
\end{proof}

\begin{proof}[Proof of \cref{thm:new-lb}.]
    Combine \cref{lem:sdp-value} with \cref{thm:perturbed-dr} to get
    \[
        K_G \geq \frac{\SDP(A_\eps)}{\val(A_\eps)} \geq \frac{1 - \lam^*}{\val(A_{DR}) - \eps(0.046 - 12 (2\eps)^{1/4})} \geq \frac{1 - \lam^*}{\val(A_{DR})} + \frac{1-\lam^*}{\val(A_{DR})^2} \cdot {\eps(0.046 - 12 (2\eps)^{1/4})}\,.
    \]
    The last inequality is using the convexity of $f(x) = \frac{1 - \lam^*}{\val(A_{DR}) - x}$ at $x=0$.
    Selecting $\eps = 4\cdot 10^{-11}$ is enough to make $0.046 - 12(2\eps)^{1/4} \geq 0.01$.
    Plug in $\lam^* \approx 0.19748$ and $\val(A_{DR}) \approx 0.4786$
    to obtain the lower bound $K_G \geq K_{DR} + 10^{-12}$.
\end{proof}

\subsection{Stability analysis}

Now we prove \cref{lem:dr-robust}.
First, we repeat the analysis of Davie and Reeds of their game.
\begin{lemma}[\cite{davie1984lower, reeds1991new}]\label{lem:Davie--Reeds}
    $\displaystyle\val(A_{DR}) \leq 4\phi(C^*)^2 - \lam^* (4\Phi(-C^*) - 1)$.
\end{lemma}
\begin{proof}
    Recall $A_{DR} = \bm{\Pi}_1 - \lam^* \bI$. We use an arbitrary $\lam \in (0,0.3)$
    in the analysis, explaining how to pick the optimal $\lam = \lam^*$ at the end.

    Let $f, g : \R^n \to \{\pm 1\}$.
    First, rotate $f$ and $g$ so that $\bm{\Pi}_1 f = \sig X_1$ for some $\sig \in \R$.
    This does not change the value of $f,g$ by \cref{lem:rotation-invariant}.
    Let $\mu = \E_{X \sim \gam} X_1 g(X)$.
    By the Cauchy--Schwarz inequality,
    \begin{align}
        \val_{A_{DR}}(f,g) = \sig \mu - \lam \E_{X \sim \gam} f(X) g(X) &\leq \left(\frac{\sig + \mu}{2}\right)^2 - \lam \E_{X \sim \gam} f(X) g(X)\,. \label{eq:cauchy-schwarz}\\
        &=\left(\E_{X \sim \gam} X_1 \left(\frac{f(X) + g(X)}{2}\right)\right)^2 - \lam \E_{X\sim \gam} f(X)g(X) \notag
    \end{align}

    Let $S = \{x \in \R^n : f(x) = g(x)\}$ and let $\gam(S)$ denote its Gaussian measure. Then:
    \begin{align}
        &\left(\E_{X \sim \gam} X_1 \left(\frac{f(X) + g(X)}{2}\right)\right)^2 \leq \left(\E_{X \sim \gam} |X_1| \cdot \bm{1}_{S}(X)\right)^2 \,, \qquad \quad \E_{X\sim \gam} f(X)g(X) = 2\gam(S) - 1\,. \label{eq:x1-correlation}
    \end{align}
    Putting these together in \cref{eq:cauchy-schwarz},
    \begin{align}\label{eq:bathtub-argument}
        \val_{A_{DR}}(f,g) \leq \left(\E_{X \sim \gam} |X_1| \cdot \bm{1}_S(X)\right)^2 - \lam \cdot (2 \gam(S) - 1)\,.
    \end{align}
    
    For fixed $\gam(S)$, the right-hand side is maximized by choosing $S$ to be of the form $S = \{x \in \R^n : |x_1| \geq C\}$.
    This is because, when the measure of $S$ is fixed, we may as well allocate the mass of $S$ to make $|X_1|$ as large as possible.
    This yields the upper bound,
    \begin{align}
        \val_{A_{DR}}(f,g) &\leq \sup_{C \geq 0} \left(\E_{X \sim \gam} |X_1| \cdot \bm{1}_{|X_1| \geq C} \right)^2 - \lam \cdot \left(4\Phi(-C) - 1\right) \notag\\
        &= \sup_{C \geq 0} 4\phi(C)^2 - \lam \cdot (4\Phi(-C)-1) \label{eq:fixed-mean}
    \end{align}
    where $\phi$ and $\Phi$ are the Gaussian PDF and CDF. Single-variable calculus can now be used to solve this optimization problem.

\begin{figure}[ht]
\label{fig:f}
\centering
\begin{tikzpicture}
\begin{groupplot}[
    group style={group size=2 by 1, horizontal sep=2cm},
    width=0.48\textwidth,
    height=0.25\textwidth,
    domain=0:4,
    samples=250,
    xmin=0, xmax=4,
    enlargelimits=false,
    tick label style={font=\tiny},
    axis line style={->, thin},
    axis x line=middle,
    axis y line=left,
    axis background/.style={fill=none},
    clip=false,
    xlabel={$C$},
    xlabel style={at={(axis description cs:0.5,-0.15)}, anchor=north},
    title style={yshift=4pt},
]

\pgfmathsetmacro{\lam}{0.19748}

\pgfmathdeclarefunction{gausspdf}{1}{%
  \pgfmathparse{1/sqrt(2*pi) * exp(-(#1)^2/2)}%
}

\pgfmathdeclarefunction{gausscdf}{1}{%
  \pgfmathparse{
    (#1 < 0) *
    (1 - (1 - 1/(1 + 0.2316419*(-#1))*(0.319381530
      + 1/(1 + 0.2316419*(-#1))*(-0.356563782
      + 1/(1 + 0.2316419*(-#1))*(1.781477937
      + 1/(1 + 0.2316419*(-#1))*(-1.821255978
      + 1/(1 + 0.2316419*(-#1))*1.330274429))))
      * gausspdf(-#1))
    )
    +
    (#1 >= 0) *
    ((1 - 1/(1 + 0.2316419*(#1))*(0.319381530
      + 1/(1 + 0.2316419*(#1))*(-0.356563782
      + 1/(1 + 0.2316419*(#1))*(1.781477937
      + 1/(1 + 0.2316419*(#1))*(-1.821255978
      + 1/(1 + 0.2316419*(#1))*1.330274429))))
      * gausspdf(#1))
    )
  }%
}

\nextgroupplot[
ylabel={$F(C)$},
    ymax=0.55,
    ymin=0.15,
]
\addplot[cyan!60!blue, thick]
{4*gausspdf(x)^2 - \lam*(4*gausscdf(-x) - 1)};

\nextgroupplot[
    ylabel={$F'(C)$},
]
\addplot[cyan!60!blue, thick]
{4*gausspdf(x)*(\lam - 2*x*gausspdf(x))};


\end{groupplot}
\end{tikzpicture}
\caption{Plots of $F(C) := 4\phi(C)^2 - \lam^* (4\Phi(-C) - 1)$ along with its first derivative, for $\lambda^* \approx 0.19748$. The only two roots of $F'$ are $C_- \approx 0.25573$ and $C_+ \approx 2.0582$.}
\end{figure}
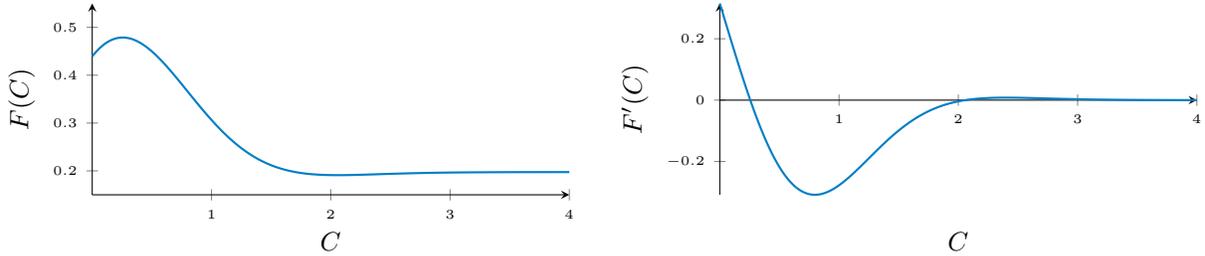
    
    Let $F(C) := 4\phi(C)^2 - \lam \cdot (4\Phi(-C) - 1)$.
    Using the properties $\Phi' = \phi$ and $\phi' = -C \phi$, the first two derivatives of $F$ with respect to $C$ are:
    \begin{align}
        F'&=8\phi\phi'+4\lambda \phi \label{eq:f-derivatives}
        =4\phi\bigl(\lambda-2C\phi\bigr)\\
        F''&=-4\lambda C\phi - 16C\phi \phi'-8\phi^2=-4\lambda C\phi+8\phi^2(2C^2-1).\notag
    \end{align}

    Hence $F'(C) = 0$ if and only if $\lam = 2C\phi(C)$.
    By taking derivatives, we see that the function $2C \phi(C)$ is strictly increasing on $(0,1)$ and decreasing on $(1,\infty)$, with maximum value at $C = 1$.
    Therefore, for all $0 < \lam < 2\phi(1) \approx 0.4839$, there are two distinct solutions to the equation $\lam = 2C\phi(C)$ which we denote $C_- \in (0,1)$ and $C_+ \in (1,\infty)$.
    The second derivative evaluated at these critical points satisfies
    \[
        F'' = -4 (2 C \phi) C \phi + 8 \phi^2 (2C^2 - 1) = 8 \phi^2 (C^2 - 1)\,.
    \]
    
    Since $C_-, C_+$ lie on opposite sides of 1, $F''(C_-) < 0$ and $C_-$ is a local maximum, while $F''(C_+) > 0$ and $C_+$ is a local minimum.
    Therefore, the maximum value of $F$ is either $F(C_-)$ or $\lim_{C \to \infty} F(C) = \lam$.
    We have
    \[
    F(C_-) \underset{(F'(0) > 0)}{>} F(0) = \frac{2}{\pi} - \lam \underset{(\lam < \frac{1}{\pi} \approx 0.318)}{>} \lam\,.
    \]
    
    So $F(C_-)$ is the maximum. We obtain the upper bound
    \begin{equation}\label{eq:f-bound}
        \val_{A_{DR}}(f,g) \leq F(C_-) = 4\phi(C_-)^2 - \lam (4\Phi(-C_-) - 1)\,.
    \end{equation}

    This proves the desired upper bound on the value of the Davie--Reeds game.
    To choose $\lam = \lam^*$ in their game, we first choose $C^* = C_- \in (0,1)$ then set $\lam^* = 2C^*\phi(C^*)$ correspondingly.
    $C^*$ is chosen to maximize the integrality gap ratio. The reciprocal of the integrality gap ratio is:
    \begin{align*}
    \frac{\val(A_\lam)}{SDP(A_\lam)} \underset{\text{(\cref{lem:sdp-value}, \cref{eq:f-bound})}}{\leq} \frac{4\phi(C)^2 - \lam(4\Phi(-C) - 1)}{1-\lam} = \frac{4\phi(C)^2 - 2C\phi(C)(4\Phi(-C) - 1)}{1-2C\phi(C)} && (=: R(C))
    \end{align*}
    The derivative with respect to $C$ of this expression is
    \[
        R'(C) = \frac{2(1-C^2) \phi(C)}{(1 - 2 C\phi(C))^2}\left(4 \phi(C)^2 - 4 \Phi(-C) + 1\right)
    \]
    This motivates us to choose $C$ to be a critical point, which satisfies $4\phi(C)^2 - 4\Phi(-C) + 1 = 0$.
\end{proof}

The optimizers of the Davie--Reeds game are characterized by checking when all of the inequalities are tight in the above proof. We prove this next.

\begin{proof}[Proof of \cref{lem:dr-optimizers}.]
    Let $f,g : \R^n \to \{\pm 1\}$ be Davie--Reeds optimizers.
    Rotate $f,g$ so that $\bm{\Pi}_1 f = \sig X_1$.
    \cref{eq:cauchy-schwarz} is tight if and only if
    \begin{equation}\label{eq:x1-coeff-equal}
        \E_{X \sim \gam} X_1 f(X) = \E_{X \sim \gam} X_1 g(X) = \sig\,.
    \end{equation}
    We will use this momentarily.

    Let $S = \{x \in \R^n : f(x) = g(x)\}$.
    \cref{eq:x1-correlation} is tight if and only if $\sign(X_1) = f(X)$ for all $X \in S$, or $\sign(X_1) = -f(X)$ for all $X\in S$ (up to a set of measure zero).
    Without of loss of generality, the former is the case, otherwise
    we apply $T \in O(n)$ which reflects $X_1$.

    The argument from \cref{eq:bathtub-argument} to \cref{eq:fixed-mean}
    is tight if and only if $S = \{x \in \R^n : |x_1| \geq C\}$
    for some $C \in \R$.
    The calculus argument is tight if and only if $C = C^*$.
    We deduce that $f,g$ satisfy item (i) of \cref{lem:dr-optimizers},
    \begin{equation}\label{eq:item-i}
        f(X) = \begin{cases}
            1 & X_1 \geq C^*\\
            -g(X) & -C^* < X_1 < C^*\\
            -1 & X_1 \leq -C^*
        \end{cases} \qquad \qquad
        g(X) = \begin{cases}
            1 & X_1 \geq C^*\\
            -f(X) & -C^* < X_1 < C^*\\
            -1 & X_1 \leq -C^*
        \end{cases}
    \end{equation}

    Towards item (ii), write
    \begin{align*}
    u(X)&:=\sign(X_1) \bm{1}_{|X_1| \geq C^*}
    \qquad
    h(X):=f(X)\bm 1_{|X_1| \leq C^*}\\
    f(X)&=u(X)+h(X) \qquad g(X)=u(X)-h(X)\\
    \implies \bm{\Pi}_1 f& = \bm{\Pi}_1 u + \bm{\Pi}_1 h \qquad \bm{\Pi}_1 g = \bm{\Pi}_1 u - \bm{\Pi}_1 h\,. \numberthis \label{eq:pi1-fg}
    \end{align*}

    We have $\bm{\Pi}_1 f = \sig X_1$ due to the initial rotation, and $\bm{\Pi}_1 u$ is also a multiple of $X_1$
    since $u$ only depends on $X_1$.
    \cref{eq:pi1-fg} implies that $\bm{\Pi}_1 g, \bm{\Pi}_1 h$ are also multiples of $X_1$.
    On the other hand, \cref{eq:x1-coeff-equal} then implies
    \begin{equation}\label{eq:pi1-fg-equal}
        \bm{\Pi}_1 f = \bm{\Pi}_1 g\,.
    \end{equation}
    Combining \cref{eq:pi1-fg,eq:pi1-fg-equal} shows $\bm{\Pi}_1 h = 0$ which proves item (ii).

    Finally, we observe that functions satisfying these properties indeed exist, for example in \cref{fig:squarewave}.
\end{proof}

To turn this into a stability argument,
we start by proving two quantitative stability lemmas.

\begin{lemma}\label{lem:bathtub-stability} \hspace{0.1cm}
        Let $S \subseteq \R^n$ and let $S^* = \{x \in \R^n : |x_1| \geq C\}$ where $C \geq 0$ is chosen so that $\gam(S) = \gam(S^*) = 2\Phi(-C)$.
        If $\E_{X \sim \gam} |X_1| \cdot \bm{1}_S(X) \geq 2\phi(C) - \eps$, then $\gam(S \symdiff S^*) \leq 4\sqrt{\eps}$.
\end{lemma}
\begin{proof}
The assumption can be rewritten as
\[
    \E_{X \sim \gam} |X_1|\cdot \bm{1}_{S^*}(X) - \E_{X \sim \gam} |X_1| \cdot \bm{1}_{S}(X) = \E_{X \sim \gam} |X_1| \cdot \bm{1}_{S^* \setminus S}(X) - \E_{X \sim \gam} |X_1| \cdot \bm{1}_{S \setminus S^*}(X) \leq \eps\,.
\]

Since $S$ and $S^*$ have the same Gaussian measure, so do $S \setminus S^*$ and $S^* \setminus S$.
Adding $\E_{X \sim \gam} C \cdot \bm{1}_{S \setminus S^*}(X) - \E_{X \sim \gam} C \cdot \bm{1}_{S^* \setminus S}(X)  = 0$ to the above equation, we get
\[
    \E_{X \sim \gam} (|X_1| - C) \cdot \bm{1}_{S^* \setminus S}(X) + \E_{X \sim \gam} (C - |X_1|) \cdot \bm{1}_{S \setminus S^*}(X) \leq \eps\,.
\]
Both terms are now non-negative, so each one is at most $\eps$. We conclude that
\[
    \E_{X \sim \gam} \abs{C - |X_1|}\cdot \bm{1}_{S \symdiff S^*}(X) \leq 2\eps\,.
\]

Now for a parameter $t \geq 0$ we bound
\[
    \gam(S \symdiff S^*) \leq \Pr(|C - |X_1|| \leq t) + \frac 1t \E_{X \sim \gam} \abs{C - |X_1|}\cdot \bm{1}_{S \symdiff S^*}(X) \leq 2\sqrt{\frac 2\pi}t + \frac{2\eps}{t}\,.
\]
Choosing $t = \sqrt{\eps}$ we get $\gam(S \symdiff S^*) \leq (2\sqrt{\frac 2\pi} + 2)\sqrt{\eps} \leq 4\sqrt{\eps}$.
\end{proof}

\begin{lemma}\label{lem:calculus-stability}
    Let $F(C) = 4 \phi(C)^2 - \lam^* (4\Phi(-C) - 1)$ and $0 \leq \eps < 0.01$.
    If $F(C) \geq F(C^*) - \eps$ then $|C - C^*| \leq 3 \sqrt{\eps}$.
\end{lemma}
\begin{proof}
    Using the derivative sign calculations in the proof of \cref{lem:Davie--Reeds} (see \cref{fig:f}), the superlevel sets of $F$ have the form
    \[
        \{C \geq 0 : F(C) \geq F(C^*) - \eps \} = [C_1, C_2] \cup [C_3, \infty)
    \]
    where $C^* = C_- \in [C_1, C_2]$ and $C_+ \leq C_3$. The maximum value of $F$ in the interval $[C_3, \infty)$ is at most $\lim_{C \to \infty}F(C) = \lam^* \approx 0.19748$ so the latter interval does not appear since $F(C^*) \approx 0.4786$.
    We bound $0 \leq C_1 \leq C_2 \leq 1/2$ using $F(0) \approx 0.4391$ and $F(1/2) \approx 0.4496$.
    
    On the interval $C \in [0,1/2]$ we bound the second derivative by, using \cref{eq:f-derivatives},
    \begin{align}
        F''(C) \leq 8 \phi(C)^2(2C^2 - 1) \leq -4\phi(1/2)^2 = -\frac{2}{\pi}e^{-1/4} < -0.49\,.\label{eq:second-derivative-bd}
    \end{align}

    Using Taylor's theorem with remainder, for every $C \in [C_1, C_2]$ there exists $\xi \in [C_1, C_2]$ such that
    \[
        F(C) = F(C^*) + \frac 12 F''(\xi)(C - C^*)^2\,.
    \]
    By \cref{eq:second-derivative-bd}, we get $F(C^*) -F(C) \geq 0.245 (C - C^*)^2$. Assuming the left side is at most $\eps$, then
    \[|C - C^*|^2 \leq \frac{\eps}{0.245} \leq 9\eps \implies |C - C^*| \leq 3\sqrt{\eps}\,.\qedhere\]
\end{proof}

\begin{proof}[Proof of \cref{lem:dr-robust}.]
Assume that $f,g : \R^n \to \{\pm 1\}$ satisfy $\val_{A_{DR}}(f,g) \geq \val(A_{DR}) - \eps$.
Following the proof of \cref{lem:Davie--Reeds}, we convert $f,g$ into a nearby pair of Davie--Reeds strip functions $f_{DR}, g_{DR}$.

First, apply a rotation so that $\bm{\Pi}_1 f = \sig X_1$.
Let $S = \{x \in \R^n : f(x) = g(x)\}$.
Partition $S = S_+ \cup S_-$ by the value of $\sign(X_1) \left(\frac{f(X) + g(X)}{2}\right)$. Then,
\begin{align*}
    & \quad \left(\E_{X \sim \gam} X_1 \left(\frac{f(X) + g(X)}{2}\right)\right)^2 \\
    &= \left(\E_{X \sim \gam} |X_1| \cdot \bm{1}_{S_+}(X) - \E_{X \sim \gam} |X_1|\cdot \bm{1}_{S_-}(X)\right)^2 \\
    &= \left(\E_{X \sim \gam} |X_1| \cdot \bm{1}_{S_+}(X) + \E_{X \sim \gam} |X_1| \cdot \bm{1}_{S_-}(X)\right)^2 - 4 \left( \E_{X \sim \gam} |X_1| \cdot \bm{1}_{S_+}(X)\right) \left(\E_{X \sim \gam} |X_1| \cdot \bm{1}_{S_-}(X)\right)\\
    &= \left(\E_{X \sim \gam} |X_1| \cdot \bm{1}_{S}(X)\right)^2 - 4 \left( \E_{X \sim \gam} |X_1| \cdot \bm{1}_{S_+}(X)\right) \left(\E_{X \sim \gam} |X_1| \cdot \bm{1}_{S_-}(X)\right)\,.
\end{align*}

Since \cref{eq:x1-correlation} has error at most $\eps$, 
we deduce that
\begin{equation}\label{eq:plus-minus-small}
    4 \left( \E_{X \sim \gam} |X_1| \cdot \bm{1}_{S_+}(X)\right) \left(\E_{X \sim \gam} |X_1| \cdot \bm{1}_{S_-}(X)\right) \leq \eps\,.
\end{equation}
Therefore, one of the inequalities holds:
\begin{equation}\label{eq:s-minus-small}
    \begin{cases}
        \displaystyle\E_{X \sim \gam} |X_1| \cdot \bm{1}_{S_+}(X) \leq \frac{\sqrt{\eps}}{2}\\
        \displaystyle\E_{X \sim \gam} |X_1| \cdot \bm{1}_{S_-}(X) \leq \frac{\sqrt{\eps}}{2}
    \end{cases}
\end{equation}
Without loss of generality, the second inequality holds (else apply $T \in O(n)$ which reflects $X_1$).

Let $S^* = \{x \in \R^n : |x_1| \geq C^*\}$.
Then since $|X_1| \geq C^*$ on $S^*$, we have
\begin{align}
    C^* \gam(S_- \cap S^*) &\leq \E_{X \sim \gam} |X_1| \cdot \bm{1}_{S_- \cap S^*}(X) \leq \E_{X \sim \gam} |X_1| \cdot \bm{1}_{S_-}(X) \underset{\text{\cref{eq:s-minus-small}}}{\leq} \frac{\sqrt{\eps}}{2} \notag\\
    \implies \gam(S_- \cap S^*) &\leq \frac{\sqrt{\eps}}{2C^*}\label{eq:s-minus-size}
\end{align}

Next, let $C \geq 0$ be such that the set $S_C = \{x \in \R^n : |x_1| \geq C\}$ has Gaussian measure $\gam(S_C) = \gam(S) = 2\Phi(-C)$.
The argument from \cref{eq:bathtub-argument,eq:fixed-mean} established that
\[
    \left(\E_{X \sim \gam}|X_1| \cdot \bm{1}_S(X)\right)^2 - \lam \cdot (2\gam(S) - 1) \leq 4\phi(C)^2 - \lam \cdot (4\Phi(-C) - 1) \leq \sup_{C^* \geq 0} 4\phi(C^*)^2 - \lam \cdot (4\Phi(-C^*) - 1)\,.
\]

Furthermore, the overall loss is at most $\eps$, so:
\begin{align}
    \label{eq:eps-1}\left(\E_{X \sim \gam}|X_1| \cdot \bm{1}_S(X)\right)^2 \geq 4\phi(C)^2 - \eps&\\
    \label{eq:eps-2} 4\phi(C)^2 - \lam\cdot(4\Phi(-C) - 1)&\geq \sup_{C^* \geq 0} 4 \phi(C^*)^2 - \lam\cdot (4\Phi(-C^*) - 1) - \eps\,.
\end{align}
Using \cref{lem:calculus-stability} on the \cref{eq:eps-2},
we have $|C - C^*| \leq 3\sqrt{\eps}$.
Then using $\Phi' = \phi$ we deduce
\begin{equation}\label{eq:alpha-stability}
    \gam(S_C \symdiff S^*) = 2|\Phi(-C) - \Phi(-C^*)| \leq 2 \left(\max_{t \in \R} \phi(t)\right) |C - C^*|  \leq 3\sqrt{\eps}\,.
\end{equation}

\noindent Factor \cref{eq:eps-1} into
\begin{equation}\label{eq:eps-3}
    \left(2\phi(C) + \E_{X \sim \gam}|X_1| \cdot \bm{1}_S(X)\right)\left(2\phi(C) - \E_{X \sim \gam}|X_1| \cdot \bm{1}_S(X)\right) \leq \eps\,.
\end{equation}
The bound $|C - C^*| \leq 3\sqrt{\eps}$ and $\eps < 0.01$ implies $\phi(C) \geq \phi(0.56) \approx 0.34105$.
Then \cref{eq:eps-1} implies $\E_{X \sim \gam}|X_1| \cdot \bm{1}_S(X) \geq 0.674$.
We deduce 
\begin{equation}
    2\phi(C) + \E_{X \sim \gam}|X_1| \cdot \bm{1}_S(X) \geq 1.01 \geq 1\,.
\end{equation}
Therefore, the square can be removed in \cref{eq:eps-3} to get,
\begin{equation}
    \E_{X \sim \gam}|X_1| \cdot \bm{1}_S(X) \geq 2\phi(C) - \eps\,.
\end{equation}
Now using \cref{lem:bathtub-stability}, we get
\begin{equation}\label{eq:s-stability}
    \gam(S \symdiff S_C) \leq 4 \sqrt{\eps}\,.
\end{equation}


Combine \cref{eq:s-stability} and \cref{eq:alpha-stability} to get
\begin{align}\label{eq:modify-2}
\gam(S \symdiff S^*) \leq 7 \sqrt{\eps}
\end{align}

We now define Davie--Reeds strip functions $f_{DR}, g_{DR} : \R^n \to \{\pm 1\}$ by
\[
    f_{DR}(X) = \begin{cases}
        \sign(X_1) & X \in S^*\\
        -f(X) & X \in S \setminus S^*\\
        f(X) & \text{otherwise}
    \end{cases} \qquad 
    g_{DR}(X) = \begin{cases}
        \sign(X_1) & X \in S^*\\
        g(X) & \text{otherwise}
    \end{cases}\,.
\]
By construction, $f_{DR}, g_{DR}$ are Davie--Reeds strip functions.
It can be seen that the modification set of both $f$ and $g$ is contained within $(S \symdiff S^*) \cup (S_- \cap S^*)$.
Using \cref{eq:s-minus-size,eq:modify-2},
we conclude that
\begin{align*}
    \|f - f_{DR}\|_2^2 &= 4 \gam(\{X : f(X) \neq f_{DR}(X)\}) \leq 4 \gam(S \symdiff S^*) + 4\gam(S_- \cap S^*) \leq 4\left(7 + \frac{1}{2C^*}\right)\sqrt{\eps}\,,\\
    \|g - g_{DR}\|_2^2 &= 4 \gam(\{X : g(X) \neq g_{DR}(X)\}) \leq 4 \gam(S \symdiff S^*) + 4\gam(S_- \cap S^*) \leq 4\left(7 + \frac{1}{2C^*}\right)\sqrt{\eps}\,.
\end{align*}
Since $C^* \approx 0.25573$, the right-hand side is at most $36\sqrt{\eps}$, and therefore
\[
    \norm{f - f_{DR}}_2 \leq 6\eps^{1/4}
    \qquad\text{and}\qquad
    \norm{g - g_{DR}}_2 \leq 6\eps^{1/4}\,.\qedhere
\]

\end{proof}

\paragraph{Acknowledgments.} G.M.\ is supported by the European Research Council through an ERC Starting Grant (Grant agreement No.~101077455, ObfusQation) and by the Deutsche Forschungsgemeinschaft (DFG, German Research Foundation) under Germany's Excellence Strategy - EXC 2092 CASA – 390781972.
C.J.\ was supported in part by the European Research Council (ERC) under the European
Union's Horizon 2020 research and innovation programme (Grant agreement No.~101019547).

\small
\bibliographystyle{alpha}
\bibliography{references.bib}

\end{document}